\documentclass[twoside, 12pt]{amsart}
\usepackage[letterpaper,hmargin=1.25in,vmargin=1.3in]{geometry}
\usepackage{amscd}
\usepackage{verbatim}
\usepackage{color}
\usepackage{comment}

\input xy
\xyoption{all}

\usepackage{amssymb}
\usepackage{hyperref}

\DeclareMathAlphabet{\cat}{OT1}{cmss}{m}{sl}

\newtheorem*{theorem*}{Theorem}
\newtheorem{theorem}{Theorem}[section]

\newtheorem{proposition}[theorem]{Proposition}
\newtheorem{lemma}[theorem]{Lemma}
\newtheorem{corollary}[theorem]{Corollary}

\theoremstyle{definition}

\newcommand{\tens}{\otimes}

\newcommand{\gmu}{\boldsymbol{\mu}}

\newcommand{\sep}{\mathrm{sep}}

\newcommand{\Ker}{\operatorname{Ker}}

\newcommand{\ind}{\operatorname{\hspace{0.3mm}ind}}

\newcommand{\Inv}{\operatorname{Inv}}

\newcommand{\dec}{\operatorname{dec}}
\newcommand{\disc}{\operatorname{disc}}

\newcommand{\Br}{\operatorname{Br}}

\newcommand{\Gal}{\operatorname{Gal}}

\newcommand{\gPGL}{\operatorname{\mathbf{PGL}}}

\newcommand{\gPGO}{\operatorname{\mathbf{PGO}}}

\newcommand{\gSL}{\operatorname{\mathbf{SL}}}
\newcommand{\gO}{\operatorname{\mathbf{O}}}

\newcommand{\gGL}{\operatorname{\mathbf{GL}}}
\newcommand{\gm}{\operatorname{\mathbb{G}}_m}

\newcommand{\gSpin}{\operatorname{\mathbf{Spin}}}
\newcommand{\gHSpin}{\operatorname{\mathbf{HSpin}}}

\newcommand{\gE}{\operatorname{\mathbf{E}}}

\newcommand{\nr}{\operatorname{nr}}

\newcommand{\Dec}{\operatorname{Dec}}

\newcommand{\norm}{\operatorname{norm}}
\newcommand{\red}{\operatorname{red}}
\newcommand{\Z}{\mathbb{Z}}

\newcommand{\QZ}{\mathop{\mathbb{Q}/\mathbb{Z}}}

\title[Degree 3 unramified cohomology of classifying spaces] % colontitle
{Degree 3 unramified cohomology of classifying spaces for exceptional groups}

\author
[S.~Baek] {Sanghoon Baek}

\address[Sanghoon Baek]{Department of Mathematical Sciences, 
	KAIST,
	291 Daehak-ro, Yuseong-gu,
	Daejeon 305-701,
	Republic of Korea}
\email{sanghoonbaek@kaist.ac.kr}
\urladdr{http://mathsci.kaist.ac.kr/~sbaek/}

\begin{document}
	
\begin{abstract}
Let $G$ be a reductive group defined over an algebraically closed field of characteristic $0$ such that the Dynkin diagram of $G$ is the disjoint union of diagrams of types $G_{2}, F_{4}, E_{6}, E_{7}, E_{8}$. We show that the degree $3$ unramified cohomology of the classifying space of $G$ is trivial. In particular, combined with articles by Merkurjev \cite{Mer17} and the author \cite{Baek}, this completes the computations of degree $3$ unramified cohomology and reductive invariants for all split semisimple groups of a homogeneous Dynkin type.
\end{abstract}

\maketitle
%\tableofcontents

\section{Introduction}

A representation $V$ of an algebraic group $G$ over a field $F$ is called $\emph{generically free}$ if there is a $G$-torsor $U\to U/G$ for a $G$-equivariant open subset $U$ of the affine variety $V$. As we can embed $G$ into the general linear group $\gGL_{n}$ for some $n$, generically free representations of $G$ always exist. The variety $U/G$ can be viewed as an algebraic approximation of the classifying space of $G$ in the sense of Totaro, which will be denoted by $BG$. By the no-name lemma, the stable rationality of $BG$ does not depend on the choice of generically free representations.

A generalized Noether's problem asks whether $BG$ is stably rational. For finite groups $G$, Swan \cite{Swan} and Saltman \cite{Sal} provided counterexamples (over $\mathbb{Q}$ and $\mathbb{C}$, respectively) to the original Noether's question. However, the generalized Noether's problem is still open for a connected algebraic group $G$ over an algebraically closed field.

A basic way to detect the stable-rationality of $BG$ is to make use of unramified cohomology of the function field $E:=F(BG)$ as follows. For a field extension $K/F$, consider the Galois cohomology group $H^{d}(K):=H^{d}\big(\Gal(K_{\sep}/K), \QZ(d-1)\big)$, where $\QZ(d-1)$ denotes the direct sum of the limit of the Galois modules $\gmu_{n}^{\tens (d-1)}$ and a $p$-part in the case $\operatorname{char}(F)=p>0$ defined via logarithmic de Rham-Witt differentials, and its $p$-primary component $H^{d}(K)\{p\}$. For every $p\neq \operatorname{char}(F)$, the subgroup $H^{d}_{\nr}(E)\{p\}$ of all unramified elements in $H^{d}(E)\{p\}$ is defined by
\[H^{d}_{\nr}(E)\{p\}=\bigcap_{v}\Ker\Big(\,\partial_{v}:H^{d}(E)\{p\}\to H^{d-1}\big(F(v)\big)\{p\}\, \Big)  \]
for all discrete valuations $v$ on $E/F$, where $\partial_{v}$ denotes the residue homomorphism. It is known that if $BG$ is a stably rational integral variety over $F$, then the group $H^{d}_{\nr}(E)
$ becomes trivial, i.e., $H^{d}_{\nr}(E)=H^{d}(F)$. Hence, any nontrivial unramified cohomology $H^{d}_{\nr}(E)\{p\}$ for some degree $d$ and some prime integer $p$ shows the non-stable rationality of $BG$. Hence, it would be necessary to determine the triviality of the group $H^{d}_{\nr}(E)$.

Now we assume that $F$ is an algebraically closed field of characteristic $0$. In \cite{Bog}, Bogomolov proved that $H^{1}_{\nr}(E)=H^{2}_{\nr}(E)=0$ for any connected group $G$. For $d=3$, the triviality of the group $H^{3}_{\nr}(E)$ has been verified in several cases. For $G=\gPGL_{n}$ (projective general linear group), the triviality of the group was proved by Saltman in \cite{Sal2}. For a simple simply connected group $G$, the same result was proved by Merkurjev (classical groups) \cite{Mer2002} and Garibaldi (exceptional groups) \cite{Skip2}. Recently, the triviality of the group $H^{3}_{\nr}(E)$ was proved for all simple groups $G$ \cite{Mer162} and all semisimple groups of types $A$, $B$, $C$, $D$ \cite{Mer17}, \cite{Baek}.

In the present paper, we determine the triviality of the degree $3$ unramified cohomology for any reductive group whose semisimple part is a group of exceptional type (see Theorem \ref{realmainthm}). 
\begin{theorem}\label{mainintro}
	Let $G$ be a reductive group over an algebraically closed field of characteristic $0$. Assume that the Dynkin diagram of $G$ is the disjoint union of diagrams of types $G_{2}, F_{4}, E_{6}, E_{7}, E_{8}$. Then, $H^{3}_{\nr}(F(BG))=0$.
\end{theorem}

We remark that if $G$ is a simple simply connected group of type $G_{2}$, $F_{4}$, $E_{6}$, or $E_{7}$ defined over the complex numbers $\mathbb{C}$, then the triviality of the group $H^{3}_{\nr}(\mathbb{C}(BG))$ immediately follows from the stronger statement that $BG$ is stably rational \cite{Bog}. However, even for a simple adjoint group $\bar{G}$ of type $E_{6}$ or $E_{7}$ it is not known whether $B\bar{G}$ is stably rational or not.

In general, the stable rationality of $BG$ for a semisimple group $G$ may be independent from isogenous groups of $G$. For instance, the classifying space $B\!\gSL_{n}$ of the special linear group $\gSL_{n}$ is stably rational for all $n$, but the classifying space $B\!\gPGL_{n}$ of its adjoint group is expected not to be stably rational for some $n$. Similarly, the space $B\!\gO^{+}_{n}$ of the special orthogonal group is stably rational for all $n$, but $B\!\gSpin_{n}$ of its simply connected group is expected not to be stably rational for some $n\geq 15$ (see \cite{Mer17prime}).

For the proof of our main theorem, we use the notion of \emph{cohomological invariants} of an algebraic group \cite{GMS}. A degree $d$ cohomological invariant of a split reductive group $G$ over a field $F$ is a morphism of functors $H^{1}(-,G)\to H^{d}(-)$ on the category of field extensions over $F$, where $H^{1}(K,G)$ denotes the set of isomorphism classes of $G$-torsors over $K$. An element in the group $\Inv^{3}(G)$ of degree $3$ invariants is \emph{normalized} if it vanishes on trivial $G$-torsors, so that such invariants forms a subgroup $\Inv^{3}_{\norm}(G)$. Hence,  $\Inv^{3}(G)=\Inv^{3}_{\norm}(G)\oplus H^{3}(F)$. We write $\Inv^{3}(G)\{p\}$ for the $p$-primary component of $\Inv^{3}(G)$. All degree $3$ normalized invariants given by a cup product of a degree $2$ invariant with a constant invariant of degree $1$ form a subgroup $\Inv^{3}_{\dec}(G)$ of $\Inv^{3}_{\norm}(G)$. The factor group $\Inv^{3}_{\ind}(G):=\Inv^{3}_{\norm}(G)/\Inv^{3}_{\dec}(G)$ is called the group of \emph{indecomposable} invariants. In particular, if $F$ is an algebraically closed field, then $\Inv^{3}_{\norm}(G)=\Inv^{3}_{\ind}(G)$ \cite{Mer164}.

A degree $3$ invariant in $\Inv^{3}(G)\{p\}$ is called \emph{unramified} if for every field extension $K/F$ all values of the invariant are contained in the group $H^{3}_{\nr}(K)\{p\}$ of all unramified elements. By \cite[Proposition 4.1]{Mer162}, the group $\Inv^{3}_{\nr}(G)\{p\}$ of all unramified invariants can be identified with the unramified cohomology group of $E=F(BG)$, i.e.,
\begin{equation}\label{rostident}
\Inv^{3}_{\nr}(G)\{p\}\simeq H^{3}_{\nr}(E)\{p\}.
\end{equation}
If $G=G_{1}\times G_{2}$ for some split semisimple groups $G_{1}$ and $G_{2}$, then by \cite[Corollary 6.3]{Mer162} we have
\begin{equation}\label{twosplitsemisimple}
\Inv^{3}(G)\simeq \Inv^{3}(G_{1})\oplus \Inv^{3}(G_{2}).
\end{equation}
As the simple simply connected groups of types of $G_{2}$, $F_{4}$, and $E_{8}$ have the trivial center and they have the trivial unramified degree $3$ cohomology groups for the corresponding classifying spaces, by  (\ref{rostident}) and (\ref{twosplitsemisimple}) the proof of our main theorem is reduced to the following (see Lemma \ref{mainpropE6} and Proposition \ref{mainpropE7}):
\begin{proposition}\label{introprop}
	Let $G$ be a split semisimple group of type $E_{6}$ or $E_{7}$ defined over an algebraically closed field $F$, i.e., $G= (\gE_{6}\times \cdots \times \gE_{6})/\gmu$ with $n\, (\geq 1)$ copies of a split simple simply connected group $\gE_{6}$ of type $E_{6}$ and a central subgroup $\gmu\subseteq \gmu_{3}^{n}$ or $G=(\gE_{7}\times \cdots \times \gE_{7})/\gmu$ with $n\, (\geq 1)$ copies of a split simple simply connected group $\gE_{7}$ of type $E_{7}$ and a central subgroup $\gmu\subseteq \gmu_{2}^{n}$. Then, for every $p\neq \operatorname{char}(F)$ we have $\Inv^{3}_{\nr}(G)\{p\}=0$.
\end{proposition}

In order to prove Proposition \ref{introprop}, we shall use the notion of \emph{reductive invariants} \cite{Mer162}. Let $G$ be a split semisimple group and let $G_{\red}$ be a split reductive group such that the commutator subgroup of $G_{\red}$ is $G$ and the center of $G_{\red}$ is a torus. Then, from the exact sequence
\begin{equation*}
1\to G\to G_{\red}\to T\to 1,
\end{equation*} 
where $T$ is a split torus, we see that $BG$ is stably birational to $BG_{\red}$, i.e., for every $p\neq \operatorname{char}(F)$, $\Inv^{3}_{\nr}(G)\{p\}=\Inv^{3}_{\nr}(G_{\red})\{p\}$. Moreover, by \cite[\S10]{Mer162} the restriction map $\Inv^{3}_{\ind}(G_{\red})\to \Inv^{3}_{\ind}(G)$ is injective and its image, denoted by $\Inv^{3}_{\red}(G)$, is called the subgroup of \emph{reductive indecomposable} invariants of $G$. Hence, we have
\begin{equation}\label{inclusionsindecomp}
	\Inv^{3}_{\nr}(G)\{p\}\subset \Inv^{3}_{\red}(G)\{p\}\subset \Inv^{3}(G)\{p\}.
\end{equation}

Assume that $G$ is a split reductive group over an algebraically closed field $F$. Then, the main result \cite[Theorem]{Mer162} shows that
\begin{equation}\label{oddprimeunramified}
\Inv^{3}_{\nr}(G)\{p\}=0 \text{ for any odd prime } p \text{ with }p\neq \operatorname{char}(F).
\end{equation}
Therefore, by (\ref{oddprimeunramified}) the statement of Proposition \ref{introprop} becomes equivalent to $\Inv^{3}_{\nr}(G)\{2\}=0$. According to (\ref{inclusionsindecomp}), it would be desirable to have the triviality of the reductive indecomposable invariants for the triviality of the unramified invariants. However, this is not the case for split semisimple groups of types $E_{6}$ and $E_{7}$ (see Propositions \ref{prop:typeE6} and \ref{prop:typeE7}).

In the proof of Proposition \ref{introprop}, a semisimple group of type $E_{6}$ can be treated by using a restriction-corestriction argument (see Lemma \ref{mainpropE6}). The main idea of the proof for a group $G$ of type $E_{7}$ is to consider a subgroup $H$ of maximal rank (of type $D_{6}\times A_{1}$) and then verify the inclusion
$\Inv^{3}(G)\{2\}\subseteq\Inv^{3}(H)\{2\}$ and the triviality of $\Inv^{3}_{\nr}(H)\{2\}$ (see Lemma \ref{Esevenlemma} and Proposition \ref{mainpropE7}). Indeed, we show that every semisimple group of a mixed Dynkin type $D_{6}\times A_{1}$ has trivial unramified degree $3$ cohomology group (Corollary \ref{keycorollary}). This approach differs from the methods used for classical groups \cite{Mer17}, \cite{Baek}, where we can describe torsors explicitly for the corresponding reductive groups.

The present paper is organized as follows. In Section \ref{Prelim} we recall the notions of cohomological invariants. In Section \ref{computationofred} we compute the reductive invariants of semisimple groups of types $E_{6}$, $E_{7}$, $D_{6}\times A_{1}$. In particular, we shall need such computations for type $D_{6}\times A_{1}$ to obtain the full description of the reductive invariants in the following section. In Section \ref{unramifiedinvariantsH}, we show the triviality of the unramified invariants of the groups of type $D_{6}\times A_{1}$. In Section \ref{finalsection}, using the results from the preceding section, we prove the main result.

\paragraph{\bf Acknowledgements.} 
This work was supported by Samsung Science and Technology
Foundation under Project Number SSTF-BA1901-02.

\section{Preliminaries on cohomological invariants}\label{Prelim}

In the present section, we introduce some notation and recall basic notions of degree $3$ invariants which will be used in the following sections. 

\subsection{Indecomposable invariants}
Let $G$ be a split semisimple group. Let $T$ be a split maximal torus of $G$ and let $T^{*}$ be the character group of $T$. Then, $G=\tilde{G}/\gmu$ for some central subgroup $\gmu$ and $\Lambda_{r}\subseteq T^{*}\subseteq \Lambda$, where $\tilde{G}$ is the corresponding simply connected cover of $G$ and $\Lambda_{r}$ (resp. $\Lambda$) denotes the weight lattice (resp. the root lattice) of $G$. Let $W$ be the Weyl group of $G$. The group $S^{2}(\Lambda)^{W}$ of $W$-invariant quadratic forms on $\Lambda$ will be denoted by $Q(G)$. Note that $Q(G\times G')=Q(G)\oplus Q(G')$ for a semisimple group $G'$.

Write $\tilde{G}=G_{1}\times \cdots \times G_{n}$ for some $n$, where $G_{i}$ is a split simple simply connected group. Let $W_{i}$ be the Weyl group of $G$ and let $\Lambda_{i}$ be the weight lattice of $G_{i}$, i.e., $W=W_{1}\times \cdots \times W_{n}$ and  $\Lambda=\bigoplus_{i=1}^{n} \Lambda_{i}$. Then, by \cite[\S3b]{Mer164} the group $Q(G_{i})=S^{2}(\Lambda_{i})^{W_{i}}$ of $W_{i}$-invariant quadratic forms on $\Lambda_{i}$ is generated by a single quadratic form $q_{i}$, i.e., 
$S^{2}(\Lambda_{i})^{W_{i}}=\Z q_{i}$ and the explicit forms of $q_{i}$ for all Dynkin types can be found in \cite[\S4]{Mer164}. Hence, every element $q$ of $Q(\tilde{G})$ can be uniquely written as $q=\sum_{i=1}^{n}d_{i}q_{i}$ for some $d_{i}\in \Z$. Let $w_{i,j}$ denote the fundamental weights of $G_{i}$. Then, the character group $T^{*}$ can be described in terms of generators $w_{i,j}$ and their relations. Using this, we can choose a $\Z$-basis of $T^{*}$, so that we can compute explicitly the subgroup $Q(G)=S^{2}(T^{*})\cap Q(\tilde{G})$ of $Q(\tilde{G})$ in terms of $q_{1},\dots, q_{n}$.

Let $Z[T^{*}]$ denote the group ring of $T^{*}$. For $\lambda\in T^{*}$, we denote by $\rho(\lambda)=\sum_{\theta\in W(\lambda)}e^{\theta}$, where $W(\lambda)$ is the $W$-orbit of $\lambda$. By \cite[\S3c]{Mer164}, the Chern class map $c_{2}:\Z[T^{*}]\to S^{2}(T^{*})$ given by $\sum_{i}e^{\lambda_{i}}\mapsto \sum_{i<j}\lambda_{i}\lambda_{j}$, $\lambda_{i}\in T^{*}$ induces a group homomorphism $c_{2}:\Z[T^{*}]^{W}\to Q(G)$. The image of the homomorphism $c_{2}$ is generated by $c_{2}\big(\rho(\lambda))=-\frac{1}{2}\sum_{\theta\in W(\lambda)}\theta^{2}$ and is denoted by $\Dec(G)$.
Then, by \cite[Theorem 3.9]{Mer164} we have
\begin{equation}\label{indqgdecg}
\Inv^{3}_{\ind}(G)\simeq Q(G)/\Dec(G).
\end{equation}
We remark that for any two semisimple groups $G$ and $G'$ we have
\begin{equation}\label{decequation}
 \Dec(G\times G')=\Dec(G)\oplus \Dec(G') \text{ and } \Dec(\bar{G})\subseteq \Dec(G)\subseteq \Dec(\tilde{G}),
\end{equation}
where $\bar{G}$ denotes the adjoint group of $G$.

By \cite[Theorem]{LM}, the isomorphism in (\ref{indqgdecg}) also holds for a split reductive group $G$ and the reductive indcomposable invariants can be computed by the following criteria.
\begin{proposition}\cite[Proposition 7.1]{LM}\label{propindecomp}
Let $G$ be a split semisimple group with the components of the Dynkin diagram of type $A$ or $D$ or $E$. Let $\alpha$ be an indecomposable invariant of $G$ corresponding to $q=\sum_{i=1}^{n}d_{i}q_{i}\in Q(G)$. Then, $\alpha$ is reductive indecomposable if and only if the order $|\bar{w}_{i,j}|$ in $\Lambda/T^{*}$ divides $d_{i}$ for all $i$ and $j$.
\end{proposition}

We shall use the following useful property for the computation of the invariants.
\begin{proposition}\cite[Proposition 7.1]{Mer162}\label{propunramified}
	Let $p$ be a prime integer and let $G$ be a split semisimple group over a field $F$. Let $\gmu$ be a central subgroup of $G$ whose order is not divisible by $p$. Then, $\Inv^{3}(G/\gmu)\{p\}\simeq \Inv^{3}(G)\{p\}$. Moreover, if $\operatorname{char}(F)\neq p$, then $\Inv^{3}_{\nr}(G/\gmu)\{p\}\simeq \Inv^{3}_{\nr}(G)\{p\}$.
\end{proposition}

\subsection{Trace forms, Clifford and Arason invariants}\label{traceformsub}
Let $(A,\sigma)$ be a pair of central simple $F$-algebra $A$ of degree $2n$ with involution $\sigma$ of the first kind. The trace form $T_{\sigma}:A\to F$ is given by $T_{\sigma}(x)=\operatorname{Trd}(\sigma(x)x)$, where $\operatorname{Trd}$ denotes the reduced trace. The restriction of $T_{\sigma}$ to $\operatorname{Sym}(A,\sigma)$ and $\operatorname{Skew}(A,\sigma)$ will be denoted by $T^{+}_{\sigma}$ and $T^{-}_{\sigma}$, respectively. Hence, $T_{\sigma}=T^{+}_{\sigma}\perp T^{-}_{\sigma}$.

Let $Q=(a,b)$ be a quaternion algebra with the canonical involution $\gamma$ and let $\sigma$ be an involution of the first kind on $M_{2n}(F)$, where $a, b\in F^{\times}$. Assume that $-1\in (F^{\times})^{2}$. Then $T_{\gamma}=\langle\langle a, b\rangle\rangle$ and $T_{\sigma}=0$ in the Witt ring $W(F)$.

Let $F$ be a field of characteristic not $2$. We write $\langle a_{1},\ldots, a_{n}\rangle$ for the diagonal quadratic form $a_{1}x_{1}^{2}+\cdots +a_{n}x_{n}^{2}$. Let $W(F)$ be the Witt ring of classes of nondegenerate quadratic forms over $F$. The kernel of the dimension morphism $\dim : W(F)\to \Z/2\Z$ is called the fundamental ideal of $W(F)$ and is denoted by $I(F)$. The $n$-th power of the fundamental ideal $I(F)$, denoted by $I^{n}(F)$, is additively generated by $n$-fold Pfister forms $\langle\langle a_{1},\ldots, a_{n}\rangle\rangle:=\langle 1, -a_{1}\rangle\tens \cdots \langle 1, -a_{n}\rangle$.

For $n\leq 3$, there are well-defined homomorphisms 
	$\boldsymbol{\mathrm{e}}_{n}: I^{n}(K)\to H^{n}(F)$ given by $\langle\langle a_{1},\ldots, a_{n}\rangle\rangle\mapsto (a_{1})\cup \cdots \cup (a_{n})$ for any field extension $K/F$. Let $C(q)$ be the Clifford algebra of an even-dimensional quadratic form $q$. The class $c(q)$ of the Clifford algebra $C(q)$ in the Brauer group $\Br(K)$ is called the Clifford invariant of $q$. Then, we have $\boldsymbol{\mathrm{e}}_{2}(q)=c(q)$ under the canonical isomorphism $H^{2}(K)\simeq \Br(K)$. In particular, the cohomological invariant $\boldsymbol{\mathrm{e}}_{3}$ is called the Arason invariant.

\subsection{Reductive invariants for type $A_{1}$} Let $G=(\gSL_{2})^{n}/\gmu$ defined over an algebraically closed field $F$ of characteristic not $2$, where $n\geq 1$ and $\gmu$ is a central subgroup of $(\gSL_{2})^{n}$, i.e., $\gmu\subset (\gmu_{2})^{n}$. Let $G_{\red}=(\gGL_{2})^{n}/\gmu$. Then, the derived subgroup of $G_{\red}$ is $G$ and the center of $G_{\red}$ is a split torus. 

Let $R$ be the subgroup of the character group $\big((\gmu_{2})^{n}\big)^{*}=\bigoplus_{i=1}^{n}(\Z/2\Z)e_{i}$ whose quotient is the character group $\gmu^{*}$. Then, for any field extension $K/F$ we have a bijection 
\begin{equation*}
H^{1}(K, G_{\red})=\{(Q_{1},\ldots, Q_{n})\,|\, \sum_{i=1}^{n}r_{i}Q_{i}=0 \text{ in } \Br(K) \text{ for all } r=(r_{i})\in R\},
\end{equation*}
where $Q_{1},\ldots, Q_{n}$ are quaternion $K$-algebras (see \cite{Mer17}). For each canonical involution $\gamma_{i}$ on $Q_{i}$, the Clifford invariant of $T_{\gamma_{i}}$ coincides with the class of $Q_{i}$ in $\Br(K)$, i.e., $c(T_{\gamma_{i}})=Q_{i}$. Hence, in $\Br(K)$ we have 
\[c(\perp_{i=1}^{n}r_{i}T_{\gamma_{i}})=\sum_{i=1}^{n}r_{i}Q_{i}=0,\]
thus $\perp_{i=1}^{n}r_{i}T_{\gamma_{i}}\in I^{3}(K)$. Therefore, the Arason invariant $\boldsymbol{\mathrm{e}}_{3}$ induces the invariants for $G_{\red}$,
\[\boldsymbol{\mathrm{e}}_{3}[r]:H^{1}(K, G_{\red})\to H^{3}(K)\]
given by $\eta=(Q_{1},\ldots, Q_{n})\mapsto \boldsymbol{\mathrm{e}}_{3}(\perp_{i=1}^{n}r_{i}T_{\gamma_{i}})$. It is shown in \cite{Mer163} that every invariant for $G_{\red}$ is of the form $\boldsymbol{\mathrm{e}}_{3}[r]$. Moreover, 

\begin{lemma}\cite[Lemma 4.3]{Mer17}\label{LemtypeA}
Let $r\in R$ with at least $3$ nonzero components. Then, the invariant $\boldsymbol{\mathrm{e}}_{3}[r]$ for $G_{\red}$ is ramified.
\end{lemma}

\section{Reductive invariants for types $E_{6}$, $E_{7}$, and $D_{6}\times A_{1}$}\label{computationofred}

In this section, we compute the group of reductive indecomposable invariants of semisimple groups of types $E_{6}$, $E_{7}$, and $D_{6}\times A_{1}$. Together with the results from \cite{Mer17} and \cite{Baek}, Propositions \ref{prop:typeE6} and \ref{prop:typeE7} finish the calculations of the reductive invariants for a semisimple group of a homogeneous Dynkin type (see also \cite[\S11]{BRZ}). In this section, we denote by $\dim(V)$ the dimension of a vector space $V$ over $\Z/2\Z$.

\begin{proposition}\label{prop:typeE6}
Let $G= (\gE_{6}\times \cdots \times \gE_{6})/\gmu$ with $n\, (\geq 1)$ copies of a split simple simply connected group $\gE_{6}$ of type $E_{6}$ and a central subgroup $\gmu\subseteq (\gmu_{3})^{n}$. Let $R$ be the subgroup of $\bigoplus_{i=1}^{n} (\Z/3\Z)e_{i}$ whose quotient is the character group $\gmu^{*}$. Then \[\Inv^{3}_{\red}(G)=(\Z/2\Z)^{n-m}\oplus (\Z/6\Z)^{m},\]
	where $m=\dim \langle e_{i}\,|\, e_{i}\in R\rangle$.  Moreover, we have
	\[\Inv^{3}_{\red}(G)\{2\}=(\Z/2\Z)^{n} \text{ and } \Inv^{3}_{\red}(G)\{3\}=(\Z/3\Z)^{m}.       \]
\end{proposition}
\begin{proof}
Let $G=(\gE_{6})^{n}/\gmu$ for some central subgroup $\gmu\simeq (\gmu_{3})^{k}$ and let $q_{1},\cdots, q_{n}$ be the corresponding invariant quadratic forms for each copy of $\gE_{6}$ in $G$. By \cite[\S 4b]{Mer163} we have $\Dec(\gE_{6})=\Dec(\bar{\gE}_{6})=6\Z q_{i}$, thus by (\ref{decequation}) we obtain 
\begin{equation}\label{decE6}
\Dec(G)=6\Z q_{1}\oplus \cdots \oplus 6\Z q_{n}.
\end{equation}

Let $T=(\gm)^{6n}/\gmu$ be the split maximal torus of $G$ and let $R$ be the subgroup of $\bigoplus_{i=1}^{n}(\Z/3\Z) e_{i}$ whose quotient is the character group $\gmu^{*}$. Then, we have 
\begin{equation}\label{relationE}
R=\{r\in \bigoplus_{i=1}^{n}(\Z/3\Z) e_{i}\,|\, f_{p}(r)=0, 1\leq p\leq k\}
\end{equation}
for some linear polynomials $f_{p}\in \Z/3\Z[t_{1},\ldots, t_{n}]$ and the following commutative diagram of exact sequences
\begin{equation}\label{Tdiagram}
\xymatrix{
	0 \ar@{->}[r] & R \ar@{->}[r] & \bigoplus_{i} (\Z/3\Z)e_{i} \ar@{->}[r]  & \gmu^{*}  \ar@{=}[d]\ar@{->}[r] & 0\\
	0 \ar@{->}[r] & T^{*} \ar@{->>}[u]\ar@{->}[r]  & \bigoplus_{i,j}\Z w_{ij}\ar@{->>}[u]^{\Phi}\ar@{->}[r]& \gmu^{*} \ar@{->}[r] & 0\\
}
\end{equation}
where $T^{*}$ is the corresponding character group and 
\begin{equation}\label{middlevE6}
\Phi\big(\sum_{i,j} a_{i,j}w_{i,j}\big)=\sum_{i} (\overline{a_{i,1}+a_{i,5}+2a_{i,3}+2a_{i,6}})e_{i}    
\end{equation}
for $1\leq i\leq n$ and $1\leq j\leq 6$, where $w_{i,j}$ denote the fundamental weights for the $i$th component of the root system of $G$. Hence, it follows from (\ref{relationE}) and (\ref{Tdiagram}) that
\begin{equation}\label{TcharacterE}
	T^{*}=\{\sum a_{i,j}w_{i,j}\,|\, f_{p}(a_{1,1}+a_{1,5}+2a_{1,3}+2a_{1,6},\ldots, a_{n,1}+a_{n,5}+2a_{n,3}+2a_{n,6})\equiv 0\}.
\end{equation}

Let us denote $|\bar{w}_{i,j}|$ the order of the fundamental weight $\bar{w}_{i,j}$ in $\Lambda/T^{*}$. Obviously, $|w_{i,j}|$ is either $1$ or $3$. Moreover, by (\ref{TcharacterE}) we obtain
\begin{equation}\label{wijorder}
	|\bar{w}_{i,j}|=3 \text{ if and only if } e_{i}\not\in R \text{ and } j=1,3,5,6.
\end{equation}
Let $\alpha$ be an indecomposable invariant of $G$ corresponding to $q=\sum_{i=1}^{n}d_{i}q_{i}\in Q(G)$. Then, by Proposition \ref{propindecomp} and (\ref{wijorder}) $\alpha$ is reductive indecomposable if and only if $3\,|\, d_{i}$ for any $i$ with $e_{i}\not\in R$. Therefore, by (\ref{decE6}) we have
\begin{equation}\label{esixeqation}
\Inv^{3}_{\red}(G)=\big(\bigoplus_{e_{i}\not\in R}\Z q_{i}/2\Z q_{i}\big)\oplus \big(\bigoplus_{e_{i}\in R}\Z q_{i}/6\Z q_{i}\big).
\end{equation}

As $\Inv^{3}_{\red}(\gE_{6})\{2\}=\Inv^{3}(\gE_{6})\{2\}=\Z/2\Z$, by Proposition \ref{propunramified} and (\ref{twosplitsemisimple}) we obtain
\begin{equation*}
\Inv^{3}_{\red}(G)\{2\}=\Inv^{3}_{\red}\big((\gE_{6})^{n}\big)\{2\}=\big(\Inv^{3}_{\red}(\gE_{6})\{2\}\big)^{n}=(\Z/2\Z)^{n}, 
\end{equation*}
thus the second statement follows by (\ref{esixeqation}).
\end{proof}

Similar to the case of type $E_{6}$, we calculate the reductive indecomposable group for tyep $E_{7}$.

\begin{proposition}\label{prop:typeE7}
Let $G=(\gE_{7}\times \cdots \times \gE_{7})/\gmu$ with $n\, (\geq 1)$ copies of a split simple simply connected group $\gE_{7}$ of type $E_{7}$ and a central subgroup $\gmu\subseteq (\gmu_{2})^{n}$. Let $R$ be the subgroup of $\bigoplus_{i=1}^{n} (\Z/2\Z)e_{i}$ whose quotient is the character group $\gmu^{*}$ and let $l=\dim R$. Then, \[\Inv^{3}_{\red}(G)=(\Z/2\Z)^{l-m}\oplus (\Z/3\Z)^{n}\oplus (\Z/4\Z)^{m},\]
where $m=\dim \langle e_{i}\,|\, e_{i}\in R\rangle$. Moreover, we have
\[\Inv^{3}_{\red}(G)\{2\}=(\Z/2\Z)^{l-m}\oplus (\Z/4\Z)^{m} \text{ and } \Inv^{3}_{\red}(G)\{3\}=(\Z/3\Z)^{n}.       \]
\end{proposition}
\begin{proof}
Let $G=(\gE_{7})^{n}/\gmu$ for some central subgroup $\gmu\simeq (\gmu_{3})^{k}$ and let $l=\dim R$ (i.e.,  $l=n-k$). Then, the invariant quadratic form $q_{i}$ for each copy of $\gE_{7}$ in $G$ is given by
\begin{equation}\label{killingformE7}
q_{i}=\sum_{j=1}^{7}w_{i,j}^{2}-w_{i,1}w_{i,3}-w_{i,2}w_{i,4}-\sum_{j=3}^{6}w_{i,j}w_{i,j+1}.
\end{equation}

Since $\Dec(\gE_{7})=\Dec(\bar{\gE}_{7})=12\Z q_{i}$ for each copy of $\gE_{7}$ (see \cite[\S 4b]{Mer163}), by (\ref{decequation}) we have 
\begin{equation}\label{decE7}
\Dec(G)=12\Z q_{1}\oplus \cdots \oplus 12\Z q_{n}.
\end{equation}

Let $T=(\gm)^{7n}/\gmu$ be the split maximal torus of $G$ and let $R$ be the subgroup of $\bigoplus_{i=1}^{n}(\Z/2\Z) e_{i}$ as in (\ref{relationE}) for some linear polynomials $f_{p}\in \Z/2\Z[t_{1},\ldots, t_{n}]$ with $1\leq p\leq k$. Then, we have the same diagram (\ref{Tdiagram}), replacing the middle vertical map (\ref{middlevE6}) by
\begin{equation*}
\Phi\big(\sum_{i,j} a_{i,j}w_{i,j}\big)=\sum_{i} (\overline{a_{i,2}+a_{i,5}+a_{i7}})e_{i},  
\end{equation*}
thus by (\ref{Tdiagram}) we obtain
\begin{equation}\label{TcharacterE7}
T^{*}=\{\sum a_{i,j}w_{i,j}\,|\, f_{p}(a_{1,2}+a_{1,5}+a_{1,7},\ldots, a_{n,2}+a_{n,5}+a_{n,7})\equiv 0 \mod 2\}.
\end{equation} Hence, we have
\begin{equation}\label{wijordereseven}
|\bar{w}_{i,j}|=2 \text{ if and only if } e_{i}\not\in R \text{ and } j=2,5,7,
\end{equation}
where $|\bar{w}_{i,j}|$ denotes the order of $\bar{w}_{i,j}$ in $\Lambda/T^{*}$.

As the group $Q(\gE_{7}^{n})$ is generated by the invariant quadratic forms in $(\ref{killingformE7})$, every element of $Q(G)$ is of the form for some $q=\sum_{i=1}^{n}d_{i}q_{i}$ for some $d_{i}\in \Z$. We apply the same argument as in \cite[\S 3.1]{Baek} together with (\ref{TcharacterE7}). Then, we have
\begin{equation*}\label{QGEseven}
Q(G)=\{\sum_{i=1}^{n}d_{i}q_{i}\,|\, f_{p}(d_{1},\ldots, d_{n})\equiv 0 \mod 4\}.
\end{equation*}
Hence, by Proposition \ref{propindecomp} and (\ref{wijordereseven}) any reductive indecomposable invariant of $G$ corresponding to $q=\sum_{i=1}^{n}d_{i}q_{i}$ satisfies
\[f_{p}(\frac{\epsilon_{1}d_{1}}{2},\ldots, \frac{\epsilon_{n}d_{n}}{2})\equiv 0 \mod 2, \text{ where }\epsilon_{i}=\begin{cases} 1 & \text{ if } e_{i}\not\in R,\\ 2 & \text{ if } e_{i}\in R,\end{cases} \]
thus, we have \begin{equation}\label{indecomreducE7}
\Inv^{3}_{\red}(G)=\frac{\{\sum_{i=1}^{n}d_{i}q_{i}\,|\, f_{p}(\frac{\epsilon_{1}d_{1}}{2},\ldots, \frac{\epsilon_{n}d_{n}}{2})\equiv 0 \mod 2\}}{\Dec(G)}.
\end{equation}

Let $R'=R\cap \big(\bigoplus_{e_{i}\not\in R}(\Z/2\Z)e_{i}\big)$. Then, the group in the numerator of (\ref{indecomreducE7}) is generated by
\begin{equation*}
\{q_{i}\,|\,e_{i}\in R\}\cup \{\sum_{i=1}^{n}2r_{i}q_{i}\,|\,r=(r_{1},\ldots, r_{n})\in R'\}\cup \{4q_{i}\,|\, e_{i}\in R'\}.
\end{equation*}
Hence, the first statement immediately follows from (\ref{decE7}).

Since $\Inv^{3}_{\red}(\gE_{7})\{3\}=\Inv^{3}(\gE_{7})\{3\}=\Z/3\Z$, it follows by Proposition \ref{propunramified} and (\ref{twosplitsemisimple}) that
\begin{equation*}
\Inv^{3}_{\red}(G)\{3\}=\Inv^{3}_{\red}\big((\gE_{7})^{n}\big)\{3\}=\big(\Inv^{3}_{\red}(\gE_{7})\{3\}\big)^{n}=(\Z/3\Z)^{n}, 
\end{equation*}
thus the second statement follows by the first statement.
\end{proof}	

Now we compute the reductive indecomposable group of type $(D_{6}\times A_{1})^{n}$, which will be used to obtain the full description of the invariants (see Proposition \ref{invariantsHred}).

\begin{proposition}\label{indecomposableH}
Let $H=(\gSpin_{12}\times \gSL_{2})^{n}/\gmu$, where $n\geq 1$ and $\gmu$ is a central subgroup. Let $R$ be the subgroup of the character group $Z:=\bigoplus_{i=1}^{3n} (\Z/2\Z)e_{i}$ of the center of $\tilde{H}$ such that $\gmu^{*}=Z/R$. Let 
\begin{equation}\label{arbi}
\bar{R}=\{(\bar{r}_{1},\bar{r}_{2},\ldots, \bar{r}_{2n-1},\bar{r}_{2n})\in \bar{Z}\,\,|\,\, \sum_{j=1}^{n}(\bar{r}_{2j-1}e_{3j-2}+\bar{r}_{2j-1}e_{3j-1}+\bar{r}_{2j}e_{3j})\in R   \},
\end{equation}
where $\bar{Z}=\bigoplus_{j=1}^{n} \big((\Z/2\Z)\bar{e}_{2j-1}\oplus (\Z/2\Z)\bar{e}_{2j}\big)$. Set 
\begin{equation}\label{arbitwo}
J_{1}=\{1\leq j\leq n \,|\, e_{3j-1}, e_{3j-2}\in R\} \text{ with } l_{1}=|J_{1}| \text{ and } \bar{R}_{1}=\bar{R}\cap Z_{1}, \text{ where }
\end{equation}
$\bar{Z}_{1}=\big(\bigoplus_{j\not\in J_{1}}(\Z/2\Z)\bar{e}_{2j-1}  \big)\oplus \big(\bigoplus_{j=1}^{n}(\Z/2\Z)\bar{e}_{2j}  \big)$. Then, we have
\begin{equation*}
\Inv^{3}(H)_{\red}=(\Z/2\Z)^{m+l_{1}-l_{2}},
\end{equation*}
where $m=\dim (\bar{R}_{1})$ and $l_{2}=\dim\langle \bar{e}_{2j},\, \bar{e}_{2p}+\bar{e}_{2q}\in \bar{R}\,\,|\,\,  \bar{e}_{2p}, \bar{e}_{2q}\not\in \bar{R}\rangle$.
\end{proposition}
\begin{proof}
Let $H=(\gSpin_{12}\times \gSL_{2})^{n}/\gmu$ for some central subgroup $\gmu\simeq (\gmu_{2})^{k}$. Let $R$ be the subgroup of the character group $\bigoplus_{i=1}^{3n} (\Z/2\Z)e_{i}$ of the center of $\tilde{H}$ whose quotient is the character group $\gmu^{*}$ and let $T=(\gm)^{3n}/\gmu$ be the split maximal torus of $H$. Then, we have $R=\{r\in \bigoplus_{i=1}^{3n}(\Z/2\Z) e_{i}\,|\, f_{p}(r)=0, 1\leq p\leq k\}$ for some linear polynomials $f_{p}\in \Z/2\Z[t_{1},\ldots, t_{3n}]$ and the same commutative diagram (\ref{Tdiagram}) of exact sequences, replacing the middle vertical map (\ref{middlevE6}) by
\begin{equation*}
\Phi: \bigoplus_{1\leq j\leq n, 1\leq s\leq 6}(\Z w_{2j-1, s}\oplus \Z w_{2j-1,s})\to \bigoplus_{i=1}^{3n} (\Z/2\Z)e_{i},
\end{equation*}
which maps $\sum_{s}( a_{2j-1,s}w_{2j-1,s}+a_{2j,1}w_{2j,1})$ to
\[A_{j}:= (\overline{a_{2j-1,1}+a_{2j-1,3}+a_{2j-1,5}})e_{3j-2}+(\overline{a_{2j-1,1}+a_{2j-1,3}+a_{2j-1,6}})e_{3j-1}+\bar{a}_{2j,1}e_{3j}.\]
Hence, we obtain 
\begin{equation}\label{TcharacterH}
T^{*}=\{\sum_{j, s}( a_{2j-1,s}w_{2j-1,s}+a_{2j,1}w_{2j,1})\,\,|\,\, f_{p}(\sum_{j} A_{j})\equiv 0 \mod 2\}.
\end{equation} 
In particular, we have $|\bar{w}_{2j,1}|=2$ if and only if $e_{3j}\not\in R$ and
\begin{equation}\label{wtwojminusone}
|\bar{w}_{2j-1,5}|=2 \text{ (resp. } |\bar{w}_{2j-1,6}|=2) \text{ if and only if } e_{3j-2}\not\in R \text{ (resp. } e_{3j-1}\not\in R),
\end{equation}
where $|\bar{w}_{2j,1}|$ (resp. $|\bar{w}_{2j-1,s}|$) denotes the order of $\bar{w}_{2j,1}$ (resp. $\bar{w}_{2j-1,s}$) in $\Lambda/T^{*}$.

By \cite[\S 4b]{Mer164}, the group $Q(\tilde{H})$ is generated by the invariant quadratic forms
\begin{equation*}
	q_{2j-1}=(\sum_{s=1}^{6} w_{2j-1, s}^{2})-(w_{2j-1, 4}w_{2j-1, 6}+\sum_{s=1}^{4}w_{2j-1,s}w_{2j-1, s+1}) \text{ and } q_{2j}=w_{2j,1}^{2}.
\end{equation*}
Then, every element of $Q(H)$ is of the form $q=\sum_{j=1}^{n}(d_{2j-1}q_{2j-1}+d_{2j}q_{2j})$ for some $d_{2j-1}, d_{2j}\in \Z$. Let $J=\{1,\ldots, n\}$, $J_{1}=\{j\in J\,|\, e_{3j-1}, e_{3j-2}\in R\}$, and $J_{3}=\{j\in J\,|\, e_{3j}\in R\}$. Then, by Proposition \ref{propindecomp} and (\ref{wtwojminusone}), an indecomposable invariant corresponding to $q$ is reductive indecomposable if and only if 
\begin{equation*}
2\,|\,d_{2j-1} \text{ for all } j\not\in J_{1}\text{ and } 2\,|\,d_{2j} \text{ for all } j\not\in J_{3}.
\end{equation*}
Applying the same argument as in \cite[\S 3.3]{Baek} and \cite[Theorem 5.6]{Baek}, we see that any reductive indecomposable invariant of $H$ corresponding to $q$ satisfies
\[\bar{f}_{p}(\frac{\epsilon_{1}d_{1}}{2},\ldots, \frac{\epsilon_{2n}d_{2n}}{2})=0 \text{ with }\epsilon_{2j-1} \text{ (resp. } \epsilon_{2j})=\begin{cases} 1 & \text{ if } j\not\in J_{1} \text{ (resp. } j\not\in J_{3}),\\ 2 & \text{ if } j\in J_{1} \text{ (resp. } j\in J_{3}),\end{cases} \]
where $\bar{f}_{p}\in \Z/2\Z[t_{1}',\ldots, t_{2n}']$ denotes the image of $f_{p}$ under the following map
\[\Z/2\Z[t_{1},\ldots, t_{3n}]\to  \Z/2\Z[t_{1}',\ldots, t_{2n}'] \text{ given by } t_{3j-2}, t_{3j-1}\mapsto t_{2j-1}', t_{3j}\mapsto t_{2j}'.    \]
Hence, we get \begin{equation}\label{indecomreducHH}
\Inv^{3}_{\red}(H)=\frac{\{\sum_{i=1}^{n}d_{i}q_{i}\,|\, \bar{f}_{p}(\frac{\epsilon_{1}d_{1}}{2},\ldots, \frac{\epsilon_{2n}d_{2n}}{2})\equiv 0 \mod 2\}}{\Dec(G)}.
\end{equation}

Now we shall compute $\Dec(H)$. Since $\Dec(\gSpin_{12})=2\Z q_{2j-1}$ (resp. $\Dec(\gSL_{2})=\Z q_{2j}$) for each copy of $\gSpin_{12}$ (resp. $\gSL_{2}$) in $\tilde{H}$ and $\Dec(\gPGO_{12}^{+})=4\Z q_{2j-1}$ (resp. $\Dec(\gPGL_{2})=4\Z q_{2j}$) for each copy of $\gPGO_{12}^{+}$ (resp. $\gPGL_{2}$) in $\bar{H}$ (see \cite[\S 4b]{Mer164}), by (\ref{decequation}) we obtain
\begin{equation}\label{basicdecH}
(4\Z q_{2j-1}\oplus 4\Z q_{2j})^{n}\subseteq \Dec(H)\subseteq (2\Z q_{2j-1}\oplus \Z q_{2j})^{n}.
\end{equation}

Let $J_{2}=\{j\in J\,|\, e_{3j-2}+e_{3j-1}\in R\}$. Then, obviously we have $J_{1}\subseteq J_{2}$, $|J_{2}|=\dim \langle \bar{e}_{2j-1}\,|\, \bar{e}_{2j-1}\in \bar{R}\rangle$, and $|J_{3}|=\dim \langle \bar{e}_{2j}\,|\, \bar{e}_{2j}\in \bar{R}\rangle$. Let $l_{3}=|J_{3}|$ and $l_{2}=\dim\langle \bar{e}_{2j},\, \bar{e}_{2p}+\bar{e}_{2q}\in \bar{R}\,\,|\,\,  \bar{e}_{2p}, \bar{e}_{2q}\not\in \bar{R}\rangle$. We show that the decomposable group $\Dec(H)$ is isomorphic to
\begin{equation}\label{decompoH}
D:=\big(\bigoplus_{j\in J_{2}} 2\Z q_{2j-1}\big)\oplus \big(\bigoplus_{j\not\in J_{2}} 4\Z q_{2j-1}\big)\oplus \big(\bigoplus_{j\in J_{3}} \Z q_{2j}\big)\oplus \big(\bigoplus_{r=1}^{l_{2}-l_{3}} 2\Z q_{2r}'\big)\oplus \big(\bigoplus_{s=1}^{n-l_{2}} 4\Z q_{2s}''\big),
\end{equation}
where each $q_{2r}'$ (resp. $q_{2s}''$) is of the form $q_{2i}+q_{2j}$ (resp. $q_{2j}$) for some $i, j\not\in J_{3}$ such that $\bigoplus_{j\not\in J_{3}}\Z q_{2j}=\big(\bigoplus_{j=1}^{l_{2}-l_{3}} \Z q_{2r}'\big)\oplus \big(\bigoplus_{j=1}^{n-l_{2}} \Z q_{2s}''\big)$. Let $D_{p}$ denote $p$-th direct summand of $D$ for $1\leq p\leq 5$.

By (\ref{TcharacterH}), $e_{3j-2}+e_{3j-1}\in R$ (resp. $e_{3j}\in R$) if and only if $w_{2j-1,1}\in T^{*}$ (resp. $w_{2j,1}\in T^{*}$). Since $c_{2}\big(\rho(w_{2j-1, 1})\big)=-2q_{2j-1}$, $c_{2}\big(\rho(w_{2j,1})\big)=-q_{2j}$, and
\begin{equation}\label{eventwolambda}
 c_{2}\big(\rho(a_{2i}w_{2i,1}+a_{2j}w_{2j,1})\big)=-2(a_{2i}^{2}q_{2i}+a_{2j}^{2}q_{2j})
\end{equation}
for any $1\leq i\neq j\leq n$  and any nonzero integers $a_{2i}$ and $a_{2j}$, we get $D_{1}\oplus D_{3}\oplus D_{4}\subseteq \Dec(H)$. By (\ref{basicdecH}), we get $D_{2}\oplus D_{5}\subseteq \Dec(H)$, thus $D\subseteq \Dec(H)$.

A character $\lambda$ in the weight lattice $\Lambda=\bigoplus_{j=1}^{n} (\Lambda_{2j-1}\oplus \Lambda_{2j})$ of $H$ can be written as 
\begin{equation*}
\lambda=\lambda_{j_{1}}+\cdots +\lambda_{j_{v}}
\end{equation*}
for some nonzero characters $\lambda_{j_{u}}\in \Lambda_{2j_{u}-1}$ or $\Lambda_{2j_{u}}$. We show that $c_{2}(\rho(\lambda))\in D$ for all $\lambda\in T^{*}$. If $v=1$ and $\lambda_{j_{1}}\in \Lambda_{2j_{1}-1}$, then as $\Dec(\gHSpin_{12})=4\Z q_{j_{1}}$, $c_{2}(\rho(\lambda))$ is contained in either $D_{1}$ or $D_{2}$. Similarly, if $v=1$ and $\lambda_{j_{1}}\in \Lambda_{2j_{1}}$, then $c_{2}(\rho(\lambda))$ is contained in either $D_{3}$ or $D_{5}$. If $v=2$, $\lambda_{j_{1}}\in \Lambda_{2j_{1}}$, and $\lambda_{j_{2}}\in \Lambda_{2j_{2}}$, then by (\ref{eventwolambda}) we get $c_{2}(\rho(\lambda))\in D_{4}$. Finally, if $v\geq 3$ or $v=2$ with $\lambda_{j_{1}}\in \Lambda_{2j_{1}-1}$, then by the action of the normal subgroups $(\Z/2\Z)^{5}$ of the Weyl group of $H$ generated by sign switching, we see that $c_{2}(\rho(\lambda))$ is divisible by $4$, thus $c_{2}(\rho(\lambda))\in D_{2}\oplus D_{5}$. Hence, $\Dec(H)\subseteq D$.

Let $\bar{R}$ be the subgroup of $\bar{Z}$ as defined in (\ref{arbi}), i.e., 
\begin{equation*}
\bar{R}=\{\bar{r}=(\bar{r}_{1},\ldots,\bar{r}_{2n})\in \bar{Z}:=\bigoplus_{i=1}^{2n}(\Z/2\Z)\bar{e}_{i}\,\,|\,\, \bar{f}_{p}(\bar{r})\equiv 0 \mod 2   \}.
\end{equation*}
Consider the following subgroups
\begin{equation*}
\bar{Z}_{2}=\big(\bigoplus_{j\not\in J_{1}}(\Z/2\Z)\bar{e}_{2j-1}  \big)\oplus \big(\bigoplus_{j\not\in J_{3}}(\Z/2\Z)\bar{e}_{2j}  \big) \text{ and } \bar{R}_{2}=\bar{R}\cap \bar{Z}_{2}
\end{equation*}
of $\bar{Z}$. Then, the group in the numerator of (\ref{indecomreducHH}) is generated by
\begin{equation*}
\{q_{2j-1}\,|\, j\in J_{1}\}\cup \{q_{2j}\,|\, j\in J_{3}\}\cup \{\sum_{i=1}^{2n}2\bar{r}_{i}q_{i}\,|\, \bar{r}=(\bar{r}_{1},\ldots, \bar{r}_{2n})\in \bar{R}_{2}\}.
\end{equation*}
Hence, by (\ref{decompoH}) we have $\Inv^{3}_{\red}(H)=(\Z/2\Z)^{\dim(\bar{R}_{2})+l_{1}+l_{3}-l_{2}}$. As $\dim(\bar{R}_{2})=\dim(\bar{R}_{1})-l_{3}$, the statement follows.\end{proof}

\begin{comment}
\begin{lemma}\label{spinslsingle}
	Let $H=(\gSpin_{12}\times \gSL_{2})/\lambda$ with a central subgroup $\lambda=\langle (1,-1,-1)\rangle \subset Z(\gSpin_{12})\times Z(\gSL_{2})=\gmu_{2}^{3}$. Then, $\Inv^{3}_{\red}(H)=\Z/2\Z$ and $\Inv^{3}(H)=\Z/4\Z$.
\end{lemma}
\begin{proof}
	Let $R$ be the subgroup of 
\end{proof}
\end{comment}

\section{Unramified invariants for type $D_{6}\times A_{1}$}\label{unramifiedinvariantsH}

In the present section, we first give an explicit description of the torsors for the corresponding reductive groups of $H$, where $H$ is a semisimple group of type $(D_{6}\times A_{1})^{n}$ as defined in Proposition \ref{indecomposableH}  (see Lemma \ref{torsorH}). Then, using Proposition \ref{indecomposableH} we provide a full description of the corresponding cohomological invariants (Proposition \ref{invariantsHred}) and show the triviality of the unramified invariants for $H$ (Corollary \ref{keycorollary}). For $n\geq 1$, we shall write $\mathbf{\Omega}_{2n}$ and $\gPGO_{2n}^{+}$ for the extended Clifford group and the projective orthogonal group, respectively (see \cite[\S13]{KMRT}).

\begin{lemma}\label{torsorH}
Let $H=(\gSpin_{12}\times \gSL_{2})^{n}/\gmu$ over a field $F$ of characteristic not $2$, where $n\geq 1$ and $\gmu$ is a central subgroup. Let $R$ be the subgroup of $Z$ whose quotient is the character group $\gmu^{*}$, where $Z:=\bigoplus_{i=1}^{3n} (\Z/2\Z)e_{i}$ denotes the character group of the center $(\gSpin_{12}\times \gSL_{2})^{n}$. Set $H_{\red}=(\mathbf{\Omega}_{12}\times \gGL_{2})^{n}/\gmu$. Then, for any field extension $K/F$ the first Galois cohomology set $H^{1}(K, H_{\red})$ is bijective to the set of $2n$-tuples
\[\big((A_{1},\sigma_{1},f_{1}), Q_{1},\ldots, (A_{j},\sigma_{j},f_{j}), Q_{j},\ldots, (A_{n},\sigma_{n}, f_{n}), Q_{n}\big)\]
of triples consisting of a central simple $K$-algebra $A_{j}$ of degree $12$ with orthogonal involution $\sigma_{j}$ of trivial discriminant and a $K$-algebra isomorphism $f_{j}: Z\big(C(A_{j},\sigma_{j})\big)\simeq K\times K$, where $Z\big(C(A_{j},\sigma_{j})\big)$ denotes the center of the Clifford algebra $C(A_{j},\sigma_{j})$, and quaternion $K$-algebras $Q_{j}$ satisfying $\sum_{j=1}^{n}C_{j}=0$ in $\Br(K)$ for all $r=(r_{1},\ldots, r_{3n})\in R$, where
\[C_{j}:=r_{3j-2}C_{j,1}+r_{3j-1}C_{j,2}+r_{3j}Q_{j}\, \text{ or }\,r_{3j-2}C_{j,2}+r_{3j-1}C_{j,1}++r_{3j}Q_{j}\]
depending on the choice of two isomorphisms $f_{j}$ for each triple $(A_{j},\sigma_{j},f_{j})$, $C_{j,1}$ and $C_{j,2}$ denote simple $K$-algebras with $C(A_{j},\sigma_{j})=C_{j,1}\times C_{j,2}$.
\end{lemma}
\begin{proof}
Let $H_{\red}=(\mathbf{\Omega}_{12}\times \gGL_{2})^{n}/\gmu$. Consider the natural exact sequence
\begin{equation*}
1\to (\gm)^{3n}/\gmu\to H_{\red}\to (\gPGO_{12}^{+}\times \gPGL_{2})^{n}\to 1.
\end{equation*}
Then, by \cite[Proposition 42]{Serre} together with Hilbert Theorem $90$ the set $H^{1}(K, H_{\red})$ is bijective to the kernel of the connecting map
\begin{equation}\label{connectingmaplemma}
\big(H^{1}(K, \gPGO_{12}^{+})\times H^{1}(K, \gPGL_{2})\big)^{n} \stackrel{\alpha}\to (\Br_{2}(K))^{3n}\stackrel{\beta}\to H^{2}\big(K,(\gmu_{2})^{3n}/\gmu\big)
\end{equation}
for any field extension $K/F$.

The first set in (\ref{connectingmaplemma}) is identified with the set of $2n$-tuples $\big((A_{j},\sigma_{j},f_{j}), Q_{j}\big)_{1\leq j\leq n}$ of triples consisting of a central simple $K$-algebra $A_{j}$ of degree $12$ with orthogonal involution $\sigma_{j}$ of trivial discriminant and a $K$-algebra isomorphism $f_{j}: Z\big(C(A_{j},\sigma_{j})\big)\simeq K\times K$, where $Z\big(C(A_{j},\sigma_{j})\big)$ denotes the center of the Clifford algebra $C(A_{j},\sigma_{j})=C_{j,1}\times C_{j,2}$, and quaternion $K$-algebras $Q_{j}$. The image of $\big((A_{j},\sigma_{j},f_{j}), Q_{j}\big)_{1\leq j\leq n}$ under the map $\alpha$ is the $3n$-tuple $(C_{j}')_{1\leq j\leq n}$ with $C_{j}':=(C_{j,1}, C_{j,2}, Q_{j})  \,\text{ or }\, (C_{j,2}, C_{j,1}, Q_{j})$ depending on the choice of two isomorphisms $f_{j}$ for each triple $(A_{j},\sigma_{j}, f_{j})$ (i.e., the image of $(A_{j},\sigma_{j}, f_{j})$  under $\alpha$ is $(C_{j,1}, C_{j,2})$ if and only if the image of $(A_{j},\sigma_{j}, f'_{j})$ for another isomorphism $f'_{j}: Z(C(A_{j}, \sigma_{j})\big)\simeq K\times K$ under $\alpha$ is $(C_{j,2}, C_{j,1})$) and the map $\beta$ is induced by the natural surjection $(\gmu_{2})^{3n}\to (\gmu_{2})^{3n}/\gmu$.

Since $(C'_{j})\in \Ker(\beta)$ if and only if it is contained in
\begin{equation*}
	\Ker\big((\Br_{2}(K))^{3n}\stackrel{\beta}\to H^{2}(K,(\gmu_{2})^{3n}/\gmu)\stackrel{r_{*}}\to H^{2}(K, \gm)\big)
\end{equation*}
for every $r\in R=\big( (\gmu_{2})^{3n}/\gmu\big)^{*}$, the statement follows.\end{proof}

Fix $j\in J_{1}=\{j\in J\,|\, e_{3j-1}, e_{3j-2}\in R\}$. Then, by Lemma \ref{torsorH} every $\eta=\big((A_{1},\sigma_{1},f_{1}), Q_{1}, \ldots, (A_{n},\sigma_{n},f_{n}), Q_{n} \big)$ in  $H^{1}(K,H_{\red})$ satisfies the relations $A_{j}=C_{j,1}=C_{j,2}=0$ in $\Br(K)$. Hence, $(A_{j},\sigma_{j},f_{j})\simeq (M_{12}(K), \sigma_{\psi_{j}})$ for some adjoint involution $\sigma_{\psi_{j}}$ with respect to a quadratic form $\psi_{j}$ with $\psi_{j}\in I^{3}(K)$. Therefore, the Arason invariant $\boldsymbol{\mathrm{e}}_{3}$ induces the following nontrivial invariant
\begin{equation}\label{projectioninvariant}
\boldsymbol{\mathrm{e}}_{3,j}: H^{1}(K,H_{\red})\to H^{3}(K)
\end{equation}
defined by $\eta\mapsto \boldsymbol{\mathrm{e}}_{3}(\psi_{j})$.

For $\bar{r}\in \bar{R}_{1}$ and $\eta=\big((A_{j},\sigma_{j},f_{j}), Q_{j}\big)\in H^{1}(K,H_{\red})$, where $\bar{R}_{1}$ is the group defined as in (\ref{arbitwo}), we define
\begin{equation}\label{phir}
\phi[\bar{r},\eta]:=\perp_{j=1}^{n}(\bar{r}_{2j-1}T^{+}_{\sigma_{j}}\perp \bar{r}_{2j}T_{\gamma_{j}}),
\end{equation}
where $T^{+}_{\sigma_{j}}$ denotes the restriction of $T_{\sigma_{j}}$ on $\operatorname{Sym}(A_{j}, \sigma_{j})$ and $\gamma_{j}$ denotes the canonical involution on $Q_{j}$.

Now we define the cohomological invariants for $H_{\red}$. Assume that $-1\in (F^{\times})^{2}$. Since both quadratic forms $T^{+}_{\sigma_{j}}$ and $T_{\gamma_{j}}$ in (\ref{phir}) have even dimension, the quadratic form $\phi[\bar{r},\eta]$ have even dimension. Since $\disc(\sigma_{j})=1$, by \cite[Proposition 11.5]{KMRT} $\disc(T^{+}_{\sigma_{j}})=1$. As $\disc(\gamma_{j})=1$, we have $\disc(\phi[\bar{r},\eta])=1$, thus $\phi[\bar{r},\eta]\in I^{2}(K)$.

By \cite[Theorem 1]{Que}, we obtain $w_{2}(T^{+}_{\sigma_{j}})=A_{j}$ and $w_{2}(T_{\gamma_{j}})=Q_{j}$ in $\Br(K)$, where $w_{2}$ denotes the Hasse invariant. Hence, we have
\begin{equation}\label{wtwohasse}
w_{2}(\phi[\bar{r},\eta])=\sum_{j=1}^{n}(\bar{r}_{2j-1}w_{2}(T^{+}_{\sigma_{j}})+\bar{r}_{2j}w_{2}(T_{\gamma_{j}}))=\sum_{j=1}^{n}(\bar{r}_{2j-1}A_{j}+\bar{r}_{2j}Q_{j}).
\end{equation}
Since $A_{j}=\bar{r}_{2j-1}C_{j,1}+\bar{r}_{2j-1}C_{j,2}$ in $\Br(K)$, where $C(A_{j}, \sigma_{j})=C_{j,1}\times C_{j,2}$, by Lemma \ref{torsorH} the last term in (\ref{wtwohasse}) becomes trivial, i.e., $w_{2}(\phi[\bar{r},\eta])=0$. Let $c:I^{2}(K)\to \Br_{2}(K)$ be the Clifford invariant. Since the Clifford invariant coincides with the Hasse invariant and $\Ker(c)=I^{3}(K)$, we obtain $\phi[\bar{r},\eta]\in I^{3}(K)$. Hence, the Arason invariant induces the following invariant for $H_{\red}$
\[\boldsymbol{\mathrm{e}}_{3}[\bar{r}]:H^{1}(K, H_{\red})\to H^{3}(K)\]
given by $\eta\mapsto \boldsymbol{\mathrm{e}}_{3}(\phi[\bar{r},\eta])$. Below, we show that every degree $3$ normalized invariant for $H_{\red}$ has the above two forms.

\begin{proposition}\label{invariantsHred}
Let $H=(\gSpin_{12}\times \gSL_{2})^{n}/\gmu$, and $H_{\red}=(\mathbf{\Omega}_{12}\times \gGL_{2})^{n}/\gmu$ defined over an algebraically closed field $F$ of characteristic not $2$, where $n\geq 1$ and $\gmu$ is a central subgroup. Then, the invariants $\boldsymbol{\mathrm{e}}_{3, j}$ and $\boldsymbol{\mathrm{e}}_{3}[\bar{r}]$ induce a surjection \[\bigoplus_{j\in J_{1}}(\Z/2\Z)\bar{e}_{2j-1}\bigoplus\bar{R}_{1}\to \Inv^{3}(H_{\red})_{\norm}\]  
given by $\bar{e}_{2j-1}\mapsto \boldsymbol{\mathrm{e}}_{3, j}$ for $j\in J_{1}$ and $\bar{r}\mapsto \boldsymbol{\mathrm{e}}_{3}[\bar{r}]$ for $\bar{r}\in \bar{R}_{1}$. Moreover, the kernel of this morphism is given by the subgroup generated by $\{\bar{e}_{2j}\,|\, \bar{e}_{2j}\in \bar{R}_{1}\}\cup \{\bar{e}_{2p}+\bar{e}_{2q}\,|\, \bar{e}_{2p}+\bar{e}_{2q}\in \bar{R}_{1}, \bar{e}_{2p}, \bar{e}_{2q}\not\in \bar{R}_{1}\}$, i.e.,
\[\Inv^{3}(H_{\red})_{\norm}\simeq \frac{\bigoplus_{j\in J_{1}}(\Z/2\Z)\bar{e}_{2j-1}\bigoplus\bar{R}_{1}}{\langle \bar{e}_{2j},\, \bar{e}_{2p}+\bar{e}_{2q}\in \bar{R}_{1}\,\,|\,\,  \bar{e}_{2p}, \bar{e}_{2q}\not\in \bar{R}_{1} \rangle}.\]
\end{proposition}
\begin{proof}
Let $R'=\langle \bar{e}_{2j},\, \bar{e}_{2p}+\bar{e}_{2q}\in \bar{R}_{1}\,\,|\,\,  \bar{e}_{2p}, \bar{e}_{2q}\not\in \bar{R}_{1} \rangle$ be the subgroup of $\bar{R}_{1}$. We first show that the for any $\bar{r}\in \bar{R}_{1}\backslash R'$ the invariant $\boldsymbol{\mathrm{e}}_{3}[\bar{r}]$ is nontrivial. Let $\bar{r}=(\bar{r}_{1},\ldots, \bar{r}_{2n})\in \bar{R}_{1}\backslash R'$. Then, we have
\[r:=(\bar{r}_{1},\bar{r}_{1}, \bar{r}_{2}, \ldots, \bar{r}_{2n-1},\bar{r}_{2n-1}, \bar{r}_{2n})\in R.\]

Let $G=(\gGL_{2})^{3n}/\gmu$. Then, by Lemma \ref{LemtypeA} there exists $\rho=(Q_{1},\ldots, Q_{3n})\in H^{1}(L, G_{\red})$ for some quaternion algebras $Q_{i}$ over a power series field $L=K((z))$ such that the image of the value of the invariant $\boldsymbol{\mathrm{e}}_{3}[r]$ for $G_{\red}$ under the residue morphism $\partial_{z}$ is nontrivial, i.e., $\partial_{z}\big(\boldsymbol{\mathrm{e}}_{3}[r](\rho)\big)\neq 0$. 

We shall find $\eta\in H^{1}(L, H_{\red})$ such that $\boldsymbol{\mathrm{e}}_{3}[\bar{r}](\eta)= \boldsymbol{\mathrm{e}}_{3}[r](\rho)$ in $H^{3}(L)$. Let
\begin{equation*}
(A_{j},\sigma_{j})=\big(M_{3}(L)\tens Q_{3j-2}\tens Q_{3j-1}, t\tens \gamma_{3j-2}\tens \gamma_{3j-1}\big),
\end{equation*}
where $t$ denotes the transpose involution on $M_{3}(L)$ and $\gamma_{3j-2}$ and $\gamma_{3j-1}$ denote the canonical involutions on $Q_{3j-2}$ and $Q_{3j-1}$. Then, the Clifford algebra of $(A_{j},\sigma_{j})$ is given by  
\begin{equation}\label{cliffordhred}
	C(A_{j},\sigma_{j})=M_{16}(Q_{3j-2})\times M_{16}(Q_{3j-1}).
\end{equation} 
Since $\rho\in H^{1}(L, G_{\red})$
it follows by Lemma \ref{torsorH} and (\ref{cliffordhred}) that $\eta:=\big((A_{j},\sigma_{j}, f_{j}), Q_{3j}\big)\in H^{1}(L, H_{\red})$ with an appropriate choice of the isomorphisms $f_{j}$. By \cite[Example 11.3]{KMRT}, we have $T^{+}_{\sigma_{j}}=T_{\gamma_{3j-2}}\perp 
T_{\gamma_{3j-1}}$ in $W(L)$, thus 
\begin{equation*}
\phi[\bar{r},\eta]=\perp_{j=1}^{n}(\bar{r}_{2j-1}T_{\gamma_{3j-2}}\perp \bar{r}_{2j-1}T_{\gamma_{3j-1}}\perp \bar{r}_{2j}T_{\gamma_{3j}}).
\end{equation*}
Hence, $\boldsymbol{\mathrm{e}}_{3}[\bar{r}](\eta)= \boldsymbol{\mathrm{e}}_{3}[r](\rho)$, i.e., the invariant $\boldsymbol{\mathrm{e}}_{3}[\bar{r}]$ is nontrivial. Since $F$ is algebraically closed, we have $\Inv^{3}(H_{\red})_{\norm}=\Inv^{3}(H_{\red})_{\ind}$. Therefore, by Proposition \ref{indecomposableH} the homomorphism $\epsilon:\bigoplus_{j\in J_{1}}(\Z/2\Z)\bar{e}_{2j-1}\bigoplus\bar{R}_{1}\to \Inv^{3}(H_{\red})_{\norm}$ given by $\bar{e}_{2j-1}\mapsto \boldsymbol{\mathrm{e}}_{3, j}$  and $\bar{r}\mapsto \boldsymbol{\mathrm{e}}_{3}[\bar{r}]$ is surjective.

Let $\bar{e}_{2j}\in \bar{R}_{1}$ and let $K/F$ be a field extension. Then, we have $e_{3j}\in R$.
Hence, by Lemma \ref{torsorH} every $\eta:=\big((A_{1},\sigma_{1}, f_{1}), Q_{1},\ldots, (A_{n},\sigma_{n}, f_{n}), Q_{n} \big)\in H^{1}(K, H_{\red})$ satisfies the relation $Q_{j}=0$ in $\Br(K)$. Therefore, the quadratic form $\phi[\bar{r},\eta]=T_{\gamma_{j}}$ becomes trivial in $W(K)$ (see Section \ref{traceformsub}). Hence, $\boldsymbol{\mathrm{e}}_{3}[\bar{e}_{2j}]\in \Ker(\epsilon)$.

Let $\bar{e}_{2p}+\bar{e}_{2q}\in \bar{R}_{1}$ with $\bar{e}_{2p},\bar{e}_{2q}\not\in \bar{R}_{1}$. Then, again by Lemma \ref{torsorH} every $\eta:=\big((A_{j},\sigma_{j}, f_{j}), Q_{j}\big)\in H^{1}(L, H_{\red})$ satisfies the relation $Q_{p}+Q_{q}=0$ in $\Br(K)$, i.e., $Q_{p}=Q_{q}$. Therefore, $\phi[\bar{r},\eta]=T_{\gamma_{p}}\perp T_{\gamma_{p}}=0$ in $W(K)$, thus $\boldsymbol{\mathrm{e}}_{3}[\bar{e}_{2p}+\bar{e}_{2q}]\in \Ker(\epsilon)$. Hence, by Proposition \ref{indecomposableH} we have  $\Ker(\epsilon)=R'$.
\end{proof}

\begin{corollary}\label{keycorollary}
Let $H=(\gSpin_{12}\times \gSL_{2})^{n}/\gmu$ defined over an algebraically closed field $F$ of characteristic not $2$, where $n\geq 1$ and $\gmu$ is a central subgroup. Then, every unramified degree $3$ invariant for $H$ is trivial, i.e., $\Inv^{3}_{\nr}(H)\{2\}=0$. 	
\end{corollary}
\begin{proof}
Let $H_{\red}=(\mathbf{\Omega}_{12}\times \gGL_{2})^{n}/\gmu$. As $BH$ is stably birational to $BH_{\red}$, it suffices to prove that every nontrivial invariant of $H_{\red}$ is ramified. For any $\bar{r}\in \bar{R}_{1}\backslash R'$, we have shown that the invariant $\boldsymbol{\mathrm{e}}_{3}[r]$ is ramified in the proof of Proposition \ref{invariantsHred}. Hence, it suffices to show that the invariant $\boldsymbol{\mathrm{e}}_{3,j}$ defined in (\ref{projectioninvariant}) is ramified. Let $j\in J_{1}$ and let $Q=(x, y)$ be a division quaternion algebra over a field extension $K/F$. Let $\psi_{j}=\langle\langle x, y, z\rangle\rangle\perp h$ be a quadratic form over $L=K((z))$, where $h$ denotes a hyperbolic form. Choose $\eta=\big((A_{1},\sigma_{1},f_{1}), Q_{1}, \ldots, (A_{n},\sigma_{n},f_{n}), Q_{n} \big)\in H^{1}(L,H_{\red})$ such that $Q_{1}=\cdots=Q_{n}=M_{2}(L)$,
\begin{equation*}
(A_{j}, \sigma_{j})=\big(M_{12}(L), \sigma_{\psi_{j}}\big), \text{ and } (A_{i}, \sigma_{i})=\big(M_{12}(L), t\big)
\end{equation*}
for all $1\leq i\neq j\leq n$, where $\sigma_{\psi_{j}}$ denotes the adjoint involution on $M_{12}(L)$ with respect to $\psi_{j}$ and $t$ denotes the transpose involution on $M_{12}(E)$. Then, $\partial_{z}\big(\boldsymbol{\mathrm{e}}_{3,j}(\eta)\big)=(x,y)\neq 0$, thus the invariant ramifies.\end{proof}

\section{Unramified invariants for exceptional groups}\label{finalsection}
	
In this section we prove our main results.

\begin{lemma}\label{mainpropE6}
Let $G=(\gE_{6})^{n}/\gmu$ defined over an algebraically closed field $F$, where $n\geq 1$ and $\gmu\subset \gmu_{3}^{n}$ is an arbitrary central subgroup. Then, for any $p\neq \operatorname{char}(F)$ we have trivial unramified degree $3$ invariants of $G$, i.e., $\Inv^{3}_{\nr}(G)\{p\}=0$.
\end{lemma}
\begin{proof}
Let $G=(\gE_{6})^{n}/\gmu$. By (\ref{oddprimeunramified}), it suffices to show that $\Inv^{3}_{\nr}(G)\{2\}=0$. As $\Inv^{3}_{\nr}(\gE_{6})\{2\}=0$ and $\gmu\subseteq \gmu_{3}^{n}$, it follows by Proposition \ref{propunramified} and (\ref{twosplitsemisimple}) that
\[\Inv^{3}_{\nr}(G)\{2\}=\Inv^{3}_{\nr}\big((\gE_{6})^{n}\big)\{2\}=\big(\Inv^{3}_{\nr}(\gE_{6})\{2\}\big)^{n}=0.\qedhere\]\end{proof}

The simply connected group $\gE_{7}$ contains a subgroup $P=(\gSpin_{12}\times \gSL_{2})/\boldsymbol{\lambda}$, where $\boldsymbol{\lambda}=\langle (1, -1, -1)\rangle\subset (\gmu_{2})^{3}$. Let $\gmu\subset (\gmu_{2})^{n}$ be a central subgroup of $(\gE_{7})^{n}$. As $P$ is a subgroup of $\gE_{7}$ of maximal rank, we have $\gmu\subset (P)^{n}$.

\begin{lemma}\label{Esevenlemma}
Let $G=(\gE_{7})^{n}/\gmu$ defined over a field $F$ of characteristic not $2$, where $\gmu\subset (\gmu_{2})^{n}$ is a central subgroup. Let $H=(P)^{n}/\gmu$. Then, $\Inv^{3}(G)\{2\}\subset \Inv^{3}(H)\{2\}$.
\end{lemma}
\begin{proof}
Let $\bar{\gE}_{7}$ be a split simple adjoint group of type $E_{7}$. Consider the commutative diagram of exact sequences
\begin{equation*}
	\xymatrix{
		H^{1}\big(F, (\gmu_{2})^{n}/\gmu\big) \ar@{->}[r]\ar@{->}[d] & H^{1}\big(F, H\big) \ar@{->}[r]\ar@{->}[d]^{\phi} & H^{1}\big(F, (P)^{n}/(\gmu_{2})^{n}\big) \ar@{->}[r]\ar@{->}[d]^{\psi} & 
		H^{2}\big(F, (\gmu_{2})^{n}/\gmu\big)  \ar@{=} [d]\\
		H^{1}\big(F, (\gmu_{2})^{n}/\gmu\big) \ar@{->}[r] & H^{1}\big(F, G\big) \ar@{->}[r]  & H^{1}\big(F, (\bar{\gE}_{7})^{n}\big)\ar@{->}[r] & H^{2}\big(F, (\gmu_{2})^{n}/\gmu\big). \\
	}
\end{equation*}
By \cite[Proposition 1]{Petrov}, the map $\psi$ is $2$-surjective (i.e., for every field $K/F$ and every  $\eta\in H^{1}\big(K, (\bar{\gE}_{7})^{n}\big)$ there exists a separable extension $E/K$ of odd degree and $\eta'\in H^{1}\big(E, (P)^{n}/(\gmu_{2})^{n}\big)$ such that $\psi(\eta')=\eta_{E}$). By diagram chasing, we see that the map $\phi$ is $2$-surjective. Hence, by \cite[Lemma 5.3]{Skip}, the restriction map $\Inv^{3}(G)\{2\}\to \Inv^{3}(H)\{2\}$ induced by $H\hookrightarrow G$ is injective.\end{proof}

\begin{proposition}\label{mainpropE7}
	Let $G=(\gE_{7})^{n}/\gmu$ defined over an algebraically closed field $F$, where $n\geq 1$ and $\gmu\subset \gmu_{2}^{n}$ is an arbitrary central subgroup. Then, for any $p\neq \operatorname{char}(F)$ we have trivial unramified degree $3$ invariants of $G$, i.e., $\Inv^{3}_{\nr}(G)\{p\}=0$.
\end{proposition}
\begin{proof}

Let $G=(\gE_{7})^{n}/\gmu$ and $H=(P)^{n}/\gmu$. By (\ref{oddprimeunramified}), it is enough to show that $\Inv^{3}_{\nr}(G)\{2\}=0$. By Corollary \ref{keycorollary} we have $\Inv^{3}_{\nr}(H)\{2\}=0$, thus by Lemma \ref{Esevenlemma}, the result follows.\end{proof}

\begin{theorem}\label{realmainthm}
Let $G$ be a reductive group over an algebraically closed field. Assume that the Dynkin diagram of $G$ is the disjoint union of diagrams of types $G_{2}$, $F_{4}$, $E_{6}$, $E_{7}$, $E_{8}$. Then, for any $p\neq \operatorname{char}(F)$ we have $\Inv^{3}_{\nr}(G)\{p\}=0$.
\end{theorem}
\begin{proof}
By (\ref{oddprimeunramified}), it suffices to show that $\Inv^{3}_{\nr}(G)\{2\}=0$. Let $\mathbf{G}_{2}, \mathbf{F}_{4}, \mathbf{E}_{8}$ denote the split simple simply connected groups  of types $G_{2}$, $F_{4}$, $E_{8}$, respectively. Since $\Inv^{3}_{\nr}(\mathbf{G}_{2})=\Inv^{3}_{\nr}(\mathbf{F}_{4})=\Inv^{3}_{\nr}(\mathbf{E}_{8})=0$ and  the groups $\mathbf{G}_{2}, \mathbf{F}_{4}, \mathbf{E}_{8}$ have the trivial center, by (\ref{twosplitsemisimple}) we may assume that the Dynkin diagram of $G$ is the disjoint union of diagrams of types $E_{6}$ and $E_{7}$. 

As the orders of the centers $\gmu_{3}$ and $\gmu_{2}$ of $\gE_{6}$ and $\gE_{7}$ are relatively prime, we may assume that $G=\big((\gE_{6})^{n_{1}}/\gmu_{1}\big)\times \big((\gE_{7})^{n_{2}}/\gmu_{2}\big)$ for some integers $n_{1}, n_{2}\geq 1$ and some central subgroups $\gmu_{1}\subset \gmu_{3}^{n_{1}}$ and $\gmu_{2}\subset \gmu_{2}^{n_{2}}$. Hence, the statement follows by Lemma \ref{mainpropE6} and Proposition \ref{mainpropE7}.\end{proof}


\begin{thebibliography}{10}

\bibitem{Baek}
S.~Baek, \emph{Degree three invariants for semisimple groups of types $B$, $C$, and $D$}, Adv. Math. \textbf{351} (2019), 195--235.


\bibitem{BRZ}
S.~Baek, R.~Devyatov, and K.~Zainoulline, \emph{The $K$-theory of versal flags and cohomological invariants of degree $3$}, Doc. Math. \textbf{22} (2017), 1117--1148.





\bibitem{Bog}
F.~A.~Bogomolov, \emph{The Brauer group of quotient spaces of linear representations}, Izv. Akad.
Nauk SSSR Ser. Mat. \textbf{51} (1987), no.~3, 485--516, 688.





\bibitem{Skip}
S.~Garibaldi, \emph{Cohomological invariants: Exceptional groups and spin groups}, Mem. Amer. Math. Soc. \textbf{200} (2009), no.~937.


\bibitem{Skip2}
S.~Garibaldi, \emph{Cohomological invariants: Exceptional groups and spin groups}, Trans. Amer. Math. Soc. \textbf{358} (2006), no.~1, 359--371.

\bibitem{GMS}
S.~Garibaldi, A.~Merkurjev, J.-P.~Serre, \emph{Cohomological Invariants in Galois Cohomology}, University Lecture Series \textbf{28}, AMS, Providence, RI, (2003).

\bibitem{KMRT} M.-A.~Knus, A.~Merkurjev, M. Rost, and J.-P. Tignol, \emph{The book of involutions}, American
Mathematical Society, Providence, RI, 1998, With a preface in French by J. Tits.

\bibitem{Mer17prime}
A.~S.~Merkurjev, \emph{Invariants of algebraic groups and retract rationality of classifying spaces},
Algebraic groups: structure and actions, 277--294, Proc. Sympos. Pure Math., \textbf{94}, Amer.
Math. Soc., Providence, RI, 2017.


\bibitem{LM}
D.~Laackman and A.~Merkurjev, \emph{Degree three cohomological invariants of reductive
	groups}, Comment. Math. Helv. \textbf{91} (2016), no.~3, 493--518.
	
\bibitem{Mer17}
A.~S.~Merkurjev, \emph{Unramified degree three invariants for reductive groups of type $A$}, J. Algebra \textbf{502} (2018), 49--60. 




\bibitem{Mer162}
A.~Merkurjev, \emph{Unramified degree three invariants of reductive groups}, Adv. Math. \textbf{293} (2016), 697--719.

\bibitem{Mer163}
A. Merkurjev, \emph{Cohomological invariants of central simple algebras}, Izvestia RAN, Ser. Mat. \textbf{80} (2016), no.~5, 869--883.

\bibitem{Mer164}
A.~Merkurjev, \emph{Degree three cohomological invariants of semisimple groups}, J. Eur. Math. Soc. \textbf{18} (2016), 657--680.

\bibitem{Mer2002}
A.~Merkurjev, \emph{Unramified cohomology of classifying varieties for classical simply connected groups}, Ann. Sci. \'{E}cole Norm. Sup. (4) \textbf{35} (2002), no.~3, 445--476.




\bibitem{Petrov}
V.~Petrov, \emph{A rational construction of {L}ie algebras of type {$E_7$}}, J. Algebra, \textbf{481} (2017), 348--361.

\bibitem{Que}
A.~Qu\'eguiner, \emph{Cohomological invariants of algebras with involution}, J. Algebra, \textbf{194} (1997), 299--330.

\bibitem{Sal}
D.~J.~Saltman, \emph{Noether’s problem over an algebraically closed field}, Invent. Math. \textbf{77} (1984) 71--84.

\bibitem{Sal2}
D.~J.~Saltman, \emph{{$H^3$} and generic matrices}, J. Algebra \textbf{195} (1997), no.~2, 387--422. 


\bibitem{Swan}
R.~G.~Swan, \emph{Invariant rational functions and a problem of Steenrod}, Invent. Math. \textbf{7} (1969) 148--158.

\bibitem{Serre}
J.-P.~Serre, \emph{Galois Cohomology}, Springer-Verlag, Berlin, 1997.


\end{thebibliography}
\end{document}